\documentclass[12pt,reqno]{amsart} 

\usepackage{times}

\usepackage[
  margin=1in
]{geometry}

\usepackage{amssymb,amsfonts,amsthm , bbm}
\usepackage{marginnote}
\usepackage{comment}
\usepackage{enumitem}
\usepackage[dvipsnames]{color}

\usepackage[colorlinks=true, pdfstartview=FitV, linkcolor=black, citecolor=blue, urlcolor=blue]{hyperref}

\def\Xint#1{\mathchoice
{\XXint\displaystyle\textstyle{#1}}%
{\XXint\textstyle\scriptstyle{#1}}%
{\XXint\scriptstyle\scriptscriptstyle{#1}}%
{\XXint\scriptscriptstyle%
\scriptscriptstyle{#1}}%
\!\int}
\def\XXint#1#2#3{{\setbox0=\hbox{$#1{#2#3}{%
\int}$ }
\vcenter{\hbox{$#2#3$ }}\kern-.6\wd0}}
\def\barint{\, \Xint -} 
\def\bariint{\barint_{} \kern-.4em \barint}
\def\bariiint{\bariint_{} \kern-.4em \barint}
\renewcommand{\iint}{\int_{}\kern-.34em \int} 
\renewcommand{\iiint}{\iint_{}\kern-.34em \int} 

\DeclareMathAlphabet{\mathcal}{OMS}{cmsy}{m}{n}

\theoremstyle{plain}

\newtheorem{theorem}{Theorem}[section]

\newtheorem{lemma}[theorem]{Lemma}

\newtheorem{corollary}[theorem]{Proposition}
\newtheorem{proposition}[theorem]{Proposition}

\theoremstyle{definition}
\newtheorem{remark}[theorem]{Remark}

\newcommand{\R}{\mathbb{R}}

\newcommand{\N}{\mathbb{N}}
\newcommand{\Z}{\mathbb{Z}}

\newcommand{\p}{\partial}
\newcommand{\la}{\langle}
\newcommand{\ra}{\rangle}
\newcommand{\les}{\lesssim}

\newcommand{\norm}[1]{\lVert #1 \rVert}
\renewcommand{\:}{\colon}

\newcommand{\wstar}{\overset{\ast}{\rightharpoonup}}
\newcommand{\into}{\hookrightarrow}

\newcommand{\uloc}{\mathrm{uloc}}

\newcommand{\loc}{{\rm loc}}

\let\div\relax
\DeclareMathOperator{\div}{div}

\let\tilde\relas
\newcommand{\tilde}[1]{\widetilde{#1}}
\DeclareMathOperator*{\esssup}{ess\,sup}

\newcommand{\BV}{{\rm BV}}
\newcommand{\TV}{{\rm TV}}

\renewcommand{\SS}{{\rm ss}}

\numberwithin{equation}{section}
\setlist[enumerate]{leftmargin=*}

\title{Long-time behavior of scalar conservation laws with critical dissipation}

\author{Dallas Albritton}
\address{Courant Institute of Mathematical Sciences, New York University, New York, NY 10012}
\email{daa399@cims.nyu.edu}
\author{Rajendra Beekie}
\address{Courant Institute of Mathematical Sciences, New York University, New York, NY 10012}
\email{beekie@cims.nyu.edu}

\begin{document}

\begin{abstract}
The critical Burgers equation $\p_t u + u \p_x u + \Lambda u = 0$ is a toy model for the competition between transport and diffusion with regard to shock formation in fluids. It is well known that smooth initial data  does not generate shocks in finite time. Less is known about the long-time behavior for `shock-like' initial data: $u_0 \to \pm a$ as $x \to \mp \infty$. We describe this long-time behavior in the general setting of multidimensional critical scalar conservation laws $\p_t u + \div f(u) + \Lambda u = 0$  when the initial data has limits at infinity. The asymptotics are given by certain self-similar solutions, whose stability we demonstrate with the optimal diffusive rates.
\end{abstract}


\maketitle


\section{Introduction}
\label{sec:introduction}

Our motivating example is the Burgers equation with critical non-local dissipation
\begin{equation}
	\label{eq:criticalburgers}
	\p_t u + u \p_x u + \Lambda u = 0
\end{equation}
and `shock-like' initial data:
\begin{equation}
  \label{eq:shocklikedata}
	u_0(x) \to \pm a \text{ as } x \to \mp \infty,
\end{equation}
where $\Lambda = (-\Delta)^{1/2}$ and $a > 0$.
This equation arises as a toy model in fluid mechanics. It models the competition between the transport non-linearity $u \p_x u$, which drives the solution towards a shock, and the dissipation term $\Lambda u$, whose smoothing effects counteract the tendency of the non-linearity to form shocks. The equation is \emph{critical} in the sense that these two terms are in balance. In PDE terms, the strongest known monotone quantities, the $L^\infty$ norm and total variation, are invariant under the scaling symmetry
\begin{equation}
  \label{eq:scalingsymmetry}
  u \to u(\lambda x, \lambda t),
\end{equation}
which preserves the equation~\eqref{eq:criticalburgers}.

By now, it is well known that solutions of~\eqref{eq:criticalburgers} evolving from smooth initial data do not form shocks in finite time. \emph{What happens in infinite time?} We answer this question for the critical Burgers equation~\eqref{eq:criticalburgers} and in the more general context of scalar conservation laws with critical non-local dissipation in $\R^n$:
\begin{equation}
	\label{eq:criticalconservationlaw}
	\p_t u + \div f(u) + \Lambda u = 0,
\end{equation}
where the initial data has `limits at infinity'. The long-time behavior is described to leading order by certain \emph{self-similar solutions}, that is, solutions invariant under the scaling symmetry~\eqref{eq:scalingsymmetry}.

Let $h \in C^\infty(S^{n-1})$ and $u_0^{\SS}(x) = h(x/|x|)$.  Let $f \: \R \to \R^n$ belong to $C^\infty_\loc(\R;\R^n)$. Let $v_0 \in L^\infty(\R^n)$ with $|v_0| \to 0$ as $|x| \to +\infty$, specifically,
\begin{equation}
\label{eq:decaycond}
	\norm{v}_{L^\infty(\R^n \setminus B_R)} \to 0 \text{ as } R \to +\infty.
\end{equation}
Let $u_0 = u_0^{\SS} + v_0$ and $\norm{u_0^\SS}_{L^\infty(\R^n)}, \norm{u_0}_{L^\infty(\R^n)} \leq m$.  Let $u^{\SS}$, $u$ be the unique \emph{entropy solutions} to~\eqref{eq:criticalconservationlaw}
with initial data $u_0^{\SS}$, $u_0$, respectively. In the context of~\eqref{eq:criticalconservationlaw}, the notion of entropy solution was introduced by Alibaud in~\cite{AlibaudEntropy2007}, and we review it below. Notice that, by virtue of its uniqueness, $u^{\SS}$ must be self similar.

Here is our main theorem:

\begin{theorem}[Long-time behavior]
\label{thm:maintheorem}
The above entropy solution $u$ converges to the self-similar solution $u^\SS$ with the following (diffusive) rates:
\begin{equation} 
  \label{eq:ourdecayrates}
	\norm{u(\cdot,t)-u^{\SS}(\cdot,t)}_{L^q(\R^n)} \les_{m,n} o_{t \to +\infty}(1) t^{\frac{n}{q}-\frac{n}{p}} \norm{u_0 - u_0^\SS}_{L^p(\R^n)}
\end{equation}
for all $1 < p \leq q \leq +\infty$. If $p = 1$, then~\eqref{eq:ourdecayrates} holds with $O(1)$ instead of $o(1)$ on the RHS.
\end{theorem}

When $f \equiv 0$,~\eqref{eq:criticalconservationlaw} reduces to the fractional heat equation, and the above diffusive rates are easily seen to be sharp.

In dimension $n=1$, we also have stability in $\BV$:
\begin{theorem}[$\BV$ convergence]
\label{thm:bvconvergence}
If also  $u_0 \in \BV(\R)$, then
\begin{equation}
  \norm{u(\cdot,t) - u^{\SS}(\cdot,t)}_{\TV(\R^n)} \to 0 \text{ as } t \to +\infty.
\end{equation}
\end{theorem}

Additionally,  $u^{\SS}$ is monotone and satisfies the following spatial asymptotics:
\begin{equation}
  \label{eq:spatialasymptotics}
  C^{-1} \la x \ra^{-1} \leq |u^\SS(x,1) - u^\SS_0| \leq C \la x \ra^{-1}, \quad |x| \geq C,
\end{equation}
provided that $u^{\SS}_0$ is not identically constant.\footnote{It may be possible to obtain more precise spatial asymptotics for $u^{\SS}$ and its derivatives by analyzing the similarity profile $u^{\SS}(\cdot,1)$, which satisfies a quasilinear nonlocal elliptic equation.

\emph{Added in print}: In the rarefaction case, qualitative properties of the self-similar profile (symmetry, monotonicity, and convexity of $u^{\SS}(\cdot,1)$, as well as asymptotics for $\partial_x u^\SS(x,1)$ as $|x| \to +\infty)$ were studied in Theorem~1.7 of~\cite{AsymptoticPropertiesEntropySolutions} from this perspective.} Notice that, when $n \geq 2$, $u^\SS_0$ does not generally belong to $\BV(\R^n)$, and the total variation is no longer scaling invariant.

\subsection*{Comparison with existing literature}

The critical Burgers equation~\eqref{eq:criticalburgers} belongs to the following family of Burgers equations with fractional diffusion:
\begin{equation}
  \label{eq:generalburgers}
  \p_t u + u \p_x u + \Lambda^s u = 0,
\end{equation}
where $s \in (0,2]$. These models were considered by Biler et al. in~\cite{BilerFractalBurgers1998}, where they are known as \emph{fractal Burgers equations}. One may consider also the analogous conservation laws with fractional diffusion $\Lambda^s$. The relevant literature is fairly extensive: 

\textit{Regularity theory}.
A detailed picture of the regularity theory of~\eqref{eq:generalburgers} was shown by Kiselev, Nazarov, and Shterenberg in~\cite{KiselevFractalBurgers2008} in the periodic setting. When $s \geq 1$, smooth initial data gives global smooth solutions, whereas when $s < 1$, solutions may develop shocks in finite time. In that case, solutions may be continued uniquely within the class of entropy solutions. See~\cite{DongDuLiFiniteTimeFractal,AlibaudOccurrenceNonAppearance} for further discussions on regularity vs. blow-up. The proof of global regularity in~\cite{KiselevFractalBurgers2008} in the critical case follows the  method of `moduli of continuity'. This method was introduced in~\cite{KiselevNazarovVolbergInventiones2007} by Kiselev, Nazarov, and Volberg in the context of the critical SQG equation:\footnote{This method has since been generalized to other models, including the 1d critical Keller-Segel equations~\cite{BurczakKellerSegel} and the 1d fractional Euler alignment system~\cite{Changhui}.}
\begin{equation}
  \label{eq:SQG}
  \tag{SQG}
  \p_t \theta + \vec{R}^\perp \theta \cdot \nabla \theta + \Lambda \theta = 0.
\end{equation}
 Other proofs of the regularity of~\eqref{eq:SQG} are contained in~\cite{CaffVassAnnals2010} (De Giorgi's method), \cite{VariationsThemeCaffVass2009}, \cite{MaekawaMiura-drift} (Nash's method), \cite{ConstVicolGAFA} (`nonlinear maximum principle'), and~\cite{ConstTarfVicolCMP}. The above proofs can be categorized as proofs of \emph{smoothing}~\cite{CaffVassAnnals2010,MaekawaMiura-drift} or \emph{propagation of regularity}~\cite{KiselevNazarovVolbergInventiones2007,VariationsThemeCaffVass2009,ConstVicolGAFA,ConstTarfVicolCMP}. The smoothing proofs notably `forget' that the equation is nonlinear. Alternative proofs of regularity for~\eqref{eq:criticalburgers}, based on smoothing, were given in~\cite{ChanCzubakBurgers2010} (De Giorgi's method) and~\cite{SilvestreHamiltonJacobi2011,SilvestreHolderEstimatesAdvectionDiffusion2012} (non-divergence form techniques). We rely on these smoothing estimates, particularly those of Silvestre, in an essential way below.\footnote{For supercritical SQG, global regularity remains open, though it is possible to show eventual regularity~\cite{SilvestreEventual,DabkowskiEventualReg,KiselevNonlocalMaximumPrinciples,ChanCzubakSilvestreEventualRegularity2010}. We mention also the recent extension of~\cite{CaffVassAnnals2010} to bounded domains in~\cite{Stokols}.}


\textit{Long-time behavior}.
The long-time behavior of~\eqref{eq:generalburgers} is perhaps less well studied than its regularity. When $s \in (0,2)$ and the initial data is well localized, the non-linearity of~\eqref{eq:generalburgers} is `irrelevant', in the sense of~\cite{BricmontKupiainenLin}, for the long-time dynamics. When $s=1$, Iwabuchi~\cite{Iwabuchi1,Iwabuchi2} demonstrated that all solutions with $u_0 \in L^1 \cap \dot B^{0}_{\infty,1}$ converge to the Poisson kernel. When $s=2$, the spaces $L^1$ and $\mathcal{M}$ (finite measures) are critical, and it is classical that the long-time behavior is given by a self-similar solution, sometimes called a \emph{diffusion wave}. This case and its precise asymptotic behavior can be illuminated by the Cole-Hopf transformation~\cite{ChernLiu,Miller,DiffusiveNWaves,BeckWayne}.

What about non-decaying solutions? The current best results in this direction concern `rarefaction-like' initial data, that is, $a < 0$ in~\eqref{eq:shocklikedata}. In~\cite{KarchConvergenceRarefactionWaves}, it was shown that such solutions converge to an inviscid rarefaction wave when $s > 1$. In~\cite{AsymptoticPropertiesEntropySolutions}, it was shown that when $s = 1$, the solutions converge to a certain self-similar solution, and when $s < 1$, the non-linearity is `irrelevant' in the long-time asymptotic expansion. Notably, in the rarefaction case, the potential term in the energy estimates for the linearized equation appears with a good sign.

In this paper, we analyze the case of `shock-like' initial data, which is less clear. Initially, one might wonder whether (i) solutions converge to a smooth traveling or standing wave, known as a `viscous shock', or perhaps (ii) solutions form a shock in infinite time. Regarding (i), it was already shown in~\cite{BilerFractalBurgers1998} that traveling wave solutions satisfying reasonable regularity conditions do not exist when $s \in (0,1]$. Regarding (ii), one might additionally wonder whether the standing waves constructed in the subcritical case $s>1$ in~\cite{ChmajTravelling} converge to a shock as $s \to 1^+$. This is apparently also not the case, as we show in Theorem~\ref{thm:maintheorem}.  

It is tempting to conjecture that, in the subcritical case $s > 1$, shock-like solutions of~\eqref{eq:generalburgers} behave as in the classical case $s=2$, where there is a unique viscous shock, whose global asymptotic stability was shown by Il{\'i}n and Ole{\u i}nik in~\cite{Oleinik}.  See~\cite{Sattinger,Kapitula,ZumbrunHoward} and many others for further developments and precise asymptotics. For $s \in (1,2)$, the uniqueness, spatial asymptotics, and global asymptotic stability of the monotone viscous shocks constructed in~\cite{ChmajTravelling} do not seem to have appeared in the literature, although local asymptotic stability was recently demonstrated in~\cite{AsymptoticTraveling}. 


\textit{Self-similarity and (non-)uniqueness}.
The two-dimensional Navier-Stokes equations
\begin{equation}
  \label{eq:2dNS}
  \tag{NS}
  \p_t \omega + u \cdot \nabla \omega - \Delta \omega = 0, \quad u = \nabla^\perp \Delta^{-1} \omega
\end{equation}
exhibit a family of self-similar solutions known as the \emph{Oseen vortices}: $\omega(x,t) = \alpha \Gamma(x,t/\nu)$, where $\Gamma$ is the heat kernel and $\alpha = \int \omega_0 \, dx$ is the circulation. In~\cite{GallayWayne}, Gallay and Wayne famously showed that all localized solutions converge to Oseen vortices as $t \to +\infty$, and, moreover, the vortex solutions are unique within a natural solution class. Our situation is analogous, with the circulation $\alpha$ corresponding to the jump parameter $a$ in~\eqref{eq:shocklikedata}. By contrast, self-similar solutions of the three-dimensional Navier-Stokes equations are expected to be non-unique~\cite{JiaSverakIllposed,guillod2017numerical}. For~\eqref{eq:SQG}, this is investigated in forthcoming work of Bradshaw and the first author. While the entropy solutions of~\eqref{eq:criticalburgers} are unique, there may be a different class of self-similar solutions with potential non-uniqueness, for example, with $u_0 \sim \log x$, so that $\nabla u_0$ is $-1$-homogeneous.

\subsection*{Main idea}

Our starting point is the existence and uniqueness of $L^{\infty}$ entropy solutions to~\eqref{eq:criticalconservationlaw}, due to Alibaud~\cite{AlibaudEntropy2007}, which immediately gives the existence and uniqueness of a self-similar solution $u^{\SS}$. Let $v = u - u^{\SS}$ be the difference between an entropy solution $u$ and the self-similar solution. Consider a sequence $(v^{(k)})_{k \in \N}$ obtained by `zooming out' on $v$ using the scaling symmetry~\eqref{eq:scalingsymmetry}. Then establishing $v(\cdot,t) \to 0$ as $t \to +\infty$ is the same as establishing $v^{(k)} \to 0$ on $\R^n \times (1/2,1)$ as $k \to +\infty$. To analyze the new problem, we exploit a key (standard) observation about viscous scalar conservation laws, namely, that $v$ satisfies the \emph{viscous continuity equation}
\begin{equation}
  \label{eq:viscouscontinuityeq}
  \p_t v + \div(gv) + \Lambda v = 0,
\end{equation}
where
\begin{equation}
  g = \frac{f(u) - f(u^{\SS})}{u - u^{\SS}} \in L^\infty(\R^n \times (0,+\infty)).
\end{equation}
In our setting, \emph{the main difficulty is that at the initial time, $g$ is no better than bounded, since $u^{\SS}_0$ is not continuous.} This is an essential feature of the problem, and we handle it using two tools: 
\begin{enumerate}
\item \emph{smoothing for drift-diffusion equations}. By the known regularity theory, the solution, which is initially merely bounded, instantaneously becomes $C^\alpha$-in-$x$. This may be bootstrapped to higher regularity. The key point is then to move the problem past $t=0$, which is done by the
\item \emph{controlled speed of propagation}. Solutions of~\eqref{eq:viscouscontinuityeq} have finite propagation speed \emph{up to the effect of the diffusion}. This allows us to keep the initial spatial decay of the solution for small positive times and exploit~\eqref{eq:viscouscontinuityeq} with smooth coefficients and smooth, decaying initial data.
\end{enumerate}
The above tools, due to~\cite{SilvestreHamiltonJacobi2011,SilvestreHolderEstimatesAdvectionDiffusion2012} and~\cite{AlibaudEntropy2007}, respectively, are key to our arguments. 
We encounter a further, technical difficulty in that the controlled speed of propagation only allows us to propagate $L^1$-based quantities. This requires the use of special norms $\norm{\cdot}_{\ell^q_k L^p_x(\R^n)}$, for example,
\begin{equation}
  \norm{v}_{\ell^\infty_k L^1_x(\R^n)} = \sup_{k \in \Z^n} \int_{k+(-1/2,1/2)^n}  | v(x) |  \, dx.
\end{equation} After the initial time, we use the smoothing effect to estimate more standard quantities, such as $\norm{v}_{L^q(\R^n)}$, in terms of these special norms.\footnote{Similar norms appear in the Navier-Stokes literature. See~\cite{bradshaw2020local} and the references therein. Apparently, these spaces are known as \emph{Wiener amalgam spaces}.} For this, we use pointwise estimates for fundamental solutions of non-local parabolic equations with subcritical lower order terms, due to Xie and Zhang~\cite{XieZhangHeatKernel2014}. When $f$ is merely Lipschitz, we offer less precise asymptotics, see Remark~\ref{rmk:lipschitzf}.\footnote{\emph{Added in print}: See Remark~\ref{rmk:alternative} for an alternative proof, due to Hongjie Dong, based on a maximal function estimate.}







\section{Preliminaries}
\label{sec:prelim}

In the sequel, constants in the $C$, $\les$ notation may implicitly depend on $n \geq 1$, $f \in W^{1,\infty}_\loc(\R;\R^n)$.

Recall that the Poisson kernel $P$ is given by
\begin{equation}
  P ( x ,  t  ) =  c_n \frac{t}{(|x|^2 + t^2 )^{\frac{n+1}{2}}},
\end{equation}
where $c_n > 0$ is chosen to satisfy $\int P(x,t) \, dx = 1$ for all $t > 0$.

In~\cite[Definition 2.3]{AlibaudEntropy2007}, Alibaud introduced the notion of \emph{entropy solution} to the critical scalar conservation law~\eqref{eq:criticalconservationlaw}. We summarize only the facts we need about entropy solutions. See Section~3 of~\cite{AlibaudEntropy2007}. For each $u_0 \in L^\infty(\R^n)$, there exists a unique entropy solution $u$ of~\eqref{eq:criticalconservationlaw}. This solution exists globally and satisfies the maximum principle
\begin{equation}
  \norm{u}_{L^\infty_{t,x}(\R^n \times (0,+\infty))} \leq \norm{u_0}_{L^\infty(\R^n)}.
\end{equation}
The PDE~\eqref{eq:criticalconservationlaw} is satisfied in the distributional sense. Finally, $u$ belongs to $C([0,T];L^1(K))$ for each $T > 0$ and compact $K \subset \R^n$.

The following proposition is contained in \cite[Theorem 3.2]{AlibaudEntropy2007}:

\begin{proposition}[Controlled speed of propagation]
\label{prop:Alibaud:formula}
Let $u_0, \tilde{u}_0 \in L^{\infty}(\mathbb{R}^n)$. Consider $u,\tilde{u}$ entropy solutions to~\eqref{eq:criticalconservationlaw} with initial conditions $u_0$ and $\tilde{u}_0$, respectively. Then for all $x_0 \in \mathbb{R}^n$, all $t > 0$ and all $R > 0$, 
\begin{equation}
\label{eq:alibaudformula}
\int_{B(x_0, R)} | u(x,t) - \tilde{u}(x,t )| \, dx\leq \int_{B(x_0, R + Lt)} P (\cdot,t ) \ast |u_0 - \tilde{u}_0|    \, dx
\end{equation}
where $L$ is the Lipschitz constant of $f$ on $[-m, m]$, with $m = \max ( \| u_0 \|_{L^{\infty}(\mathbb{R}^n)} , \| \tilde{u}_0 \|_{L^{\infty}(\mathbb{R}^n)} 
)$. 
\end{proposition}

We use Proposition~\ref{prop:Alibaud:formula} to establish

\begin{corollary}[Controlled $\BV$]
\label{cor:controlledbv}
If $u_0 \in \BV(\R^n)$, then $u(\cdot,t) \in \BV(\R^n)$ with $\norm{u(\cdot,t)}_{\TV(\R^n)} \leq \norm{u_0}_{\TV(\R^n)}$ for all $t > 0$.
Let $\psi \in C^\infty_0(\R^n)$ be non-negative and radial with $\psi \equiv 1$ in a neighborhood of the origin. Let $\psi(x,t) = \psi(x - xLt/|x|)$ when $|x| \geq Lt$ and $\psi \equiv 1$ otherwise.
Then, for all $x_0 \in \R^n$, all $t > 0$, and $k = 1,\hdots,n$,
\begin{equation}
\label{eq:alibaudformulaforbv}
\int_{\R^n} \psi(x-x_0) |\omega_k (x,t)| \, dx \leq \int_{\R^n} \psi(x-x_0,t) P (\cdot,t ) \ast |\omega_{k,0}|    \, dx,
\end{equation}
where $\omega_k = \p_k u$ and $\omega_{k,0} = \p_k \omega_0$ are finite measures.\footnote{This is why we require integration against continuous $\psi$ on the LHS.}
\end{corollary}

 By approximation, if also $\nabla u(\cdot,t) \in L^1(\R^n)$ for a given $t > 0$, then $\psi = \mathbf{1}_{B_R}$ with $R > 0$ is an admissible weight function.

\begin{proof}
The global $\BV$ bound is directly from \cite[Proposition~3.4]{AlibaudEntropy2007}. Let us justify~\eqref{eq:alibaudformulaforbv} with $x_0 = 0$ when $\nabla u_0 \in L^1(\R^n)$ is compactly supported. 
First, Alibaud's formula~\eqref{eq:alibaudformula} holds with integration against weight $\psi$ as in~\eqref{eq:alibaudformulaforbv}. This is shown by integrating~\eqref{eq:alibaudformula} according to the principle $\int_{\R^n} \psi F \, dx = \int_0^\infty \int_{ \{ \psi > \lambda \} } F \, dx \, d\lambda$. Let $D_k^\varepsilon$ be the different quotient operator $D_k^\varepsilon u = (u(x) - u(x-\varepsilon \vec{e}_k))/\varepsilon$. Letting $\tilde{u} = u(\cdot-\varepsilon \vec{e}_k,t)$ in~\eqref{eq:alibaudformula} with weight $\psi$, and dividing by $\varepsilon$, we have
\begin{equation}
\int_{\R^n} \psi(x) | D_k^\varepsilon u(x,t) | \, dx \leq \int_{\R^n} \psi(x,t) P (\cdot,t ) \ast |D_k^\varepsilon u_0|    \, dx.
\end{equation}
We have $D_k^\varepsilon u_0 \to \omega_{k,0}$ strongly in $L^1(\R^n)$. Then $P(\cdot,t) \ast |D_k^\varepsilon u_0| \to P(\cdot,t) \ast |\omega_{k,0}|$ in $L^1(\R^n)$ also. This implies that the LHS remains  bounded as $\varepsilon \to 0^+$. Hence, $\nabla u(\cdot,t)$ actually belongs to $L^1(\R^n)$, and the LHS converges to $\int_{\R^n} \psi(x) |\omega_k(x,t)| \, dx$ as $\varepsilon \to 0^+$. To complete the proof for general $u_0 \in \BV(\R^n)$, we approximate $u_0$ in $L^\infty(\R^n)$ by $u^{(i)}_0$, $i \in \N$, and we approximate $\nabla u_0$ weakly-$\ast$ in $\mathcal{M}(\R^n)$ by $\nabla u_0^{(i)}$ compactly supported, belonging to $L^1(\R^n)$, and satisfying $|\omega_{k,0}^{(i)}| \wstar |\omega_{k,0}|$ in the sense of measures. Then one may verify, using the Lebesgue dominated convergence theorem and kernel estimates, that $P(\cdot,t) \ast |\omega_{k,0}^{(i)}| \to P(\cdot,t) \ast |\omega_{k,0}|$  strongly in $L^1(\R^n)$. The LHS is handled by lower semicontinuity. This completes the proof.
\end{proof}

Proposition~\ref{prop:Alibaud:formula} only allows us to propagate $L^1$-based quantities, which then smooth to $L^q$-based quantities, $q \geq 1$, after the initial time:

Let $\ell > 0$ and $\ell \, \square(k)$ be the open cube with center at $k$ and side length $\ell$. That is, $\ell \, \square(k) = k + (-\ell/2,\ell/2)^n$. We write $\square(k) = 1 \square(k)$. Define
\begin{equation}
  \norm{f}_{\ell^p_k L^q_x(\R^n)} = \left\lVert \left\lVert f \right\rVert_{L^q_x(\square(k))} \right\rVert_{\ell^p_k(\Z^n)}.
\end{equation}
When $p=+\infty$, the space $\ell^\infty_k L^q_x(\R^n)$ is known in the literature as $L^q_\uloc(\R^n)$. We have $L^p(\R^n) = \ell^p_k L^p_x(\R^n)$ with equality of norms. We also have the obvious embeddings
\begin{equation}
  \label{eq:trivialembedding1}
  \norm{f}_{\ell^p_k L^{q_1}_x(\R^n)} \leq \norm{f}_{\ell^p_k L^{q_2}_x(\R^n)}
\end{equation}
when $q_1 \leq q_2$, and
\begin{equation}
  \norm{f}_{\ell^{p_2}_k L^q_x(\R^n)} \leq \norm{f}_{\ell^{p_1}_k L^q_x(\R^n)}
\end{equation}
when $p_1 \leq p_2$. The short-time and small-distance behavior of these spaces is akin to that of $L^q(\R^n)$, whereas the large-distance behavior is more closely akin to that of $L^p(\R^n)$.

We will require the following smoothing estimates when $q=1$ or $p=q$. However, it is no more effort to prove the general estimates:
\begin{lemma}[Smoothing for the heat equation]
\label{lem:smoothing}
Let $p,q_1,q_2 \in [1,+\infty]$ with $q_1 \leq q_2$. Let $w_0 \in \ell^p_k L^{q_1}_x(\R^n)$. Define
\begin{equation}
  w(\cdot,t) = P(\cdot,t) \ast w_0.
\end{equation}
Then for $t \leq 1$,
\begin{equation}
  \norm{w(\cdot,t)}_{\ell^p_k L^{q_2}_x(\R^n)} \les t^{\frac{n}{q_2} - \frac{n}{q_1}} \norm{w_0}_{\ell^p_k L^{q_1}_x(\R^n)}.
\end{equation}
\end{lemma}

\begin{proof}
Let $k \in \Z^n$. We decompose $\mathbb{R}^n$ into near-to-$k$ and far-from-$k$ regions:
\begin{equation}
\begin{aligned}
  &\norm{w(x,t)}_{L^{q_2}_x(\square(k))} \leq  \underbrace{\left\lVert \int_{y \in 3\square(k)} P(x-y,t) |w_0|(y) \, dy \right\rVert_{L^{q_2}_x(\square(k))}}_{=I_1(k)} \\
  &\quad + \underbrace{\left\lVert \int_{y \in \R^n \setminus 3\square(k)} P(x-y,t) |w_0|(y) \, dy \right\rVert_{L^{q_2}_x(\square(k))}}_{=I_2(k)}.
  \end{aligned}
\end{equation}
Eventually, we will sum in $\ell^p_k(\Z^n)$. First, we estimate $I_1(k)$:
\begin{equation}
  I_1(k) \leq \left\lVert \sum_{|j|_{\infty} \leq 1} P(\cdot,t) \ast (\mathbf{1}_{\square(k+j)} |w_0|) \right\rVert_{L^{q_2}_x(\R^n)} \les t^{\frac{n}{q_2}-\frac{n}{q_1}} \sum_{|j|_{\infty} \leq 1} \norm{w_0}_{L^{q_1}_x(\square(k+j))},
\end{equation}
by Young's convolution inequality.
We now sum in $\ell^p_k(\Z^n)$. By the triangle inequality, and since there are only a finite number of boxes (specifically, $3^n$) in the $j$ sum, we have
\begin{equation}
  \norm{I_1(k)}_{\ell^p_k(\Z^n)} \les t^{\frac{n}{q_2} - \frac{n}{q_1}} \norm{w_0}_{\ell^p_k L^{q_1}_x(\R^n)}.
\end{equation}
Now we estimate $I_2(k)$:
\begin{equation}
\begin{aligned}
    I_2(k) 
        &\leq \sum_{|j|_{\infty} > 1}  \int_{\square(k - j) } |w_0(y)| \, \norm{ P(x -y,t) }_{L^{q_2}_x(\square(k))} \, dy \\
        &\leq \sum_{|j|_{\infty} > 1} \left[ \sup_{x\in\square(k)} \sup_{y \in \square(k-j)} P(x-y,t) \right]  \int_{\square(k -j)} |w_0(y)| \, dy,
\end{aligned}
\end{equation}
where we used H{\"o}lder's inequality in $x$ and $|\square(k)|=1$.
When $x \in \square(k)$ and $y \in \square(k - j)$, we have $|x-y| \geq |j| - 1$. Recall that $P(z,t) \les t/|z|^{n + 1} \les 1/|z|^{n+1}$ for $t \leq 1$. Hence,
\begin{equation}
  \sup_{x\in\square(k)} \sup_{y \in \square(k-j)} P(x-y,t) \les \frac{1}{(|j|-1)^{n+1}},
\end{equation}
and
\begin{equation}
\label{eq:i2est2}
  I_2(k) \les \sum_{|j|_{\infty} > 1} \frac{1}{(|j|-1)^{n+1}} \int_{\square(k -j)} |w_0(y)| \, dy.
\end{equation}
One may recognize~\eqref{eq:i2est2} as a discrete convolution with a summable-in-$j$ kernel.
Applying $\norm{\cdot}_{\ell^p_k(\Z^n)}$, we have
\begin{equation}
  \label{eq:i2est3}
  \norm{I_2(k)}_{\ell^p_k(\Z^n)} \les \sum_{|j|_{\infty} > 1} \frac{1}{(|j|-1)^{n+1}} \norm{w_0}_{\ell^p_k L^1_x(\R^n)} \les \norm{w_0}_{\ell^p_k L^{q_1}_x(\R^n)},
\end{equation}
where we use the trivial embedding~\eqref{eq:trivialembedding1}.
This completes the proof.
\end{proof}

We now justify that the entropy solutions immediately become H{\"o}lder continuous and better:

\begin{proposition}[Regularity]
\label{pro:regularity}
Let $u$ be the unique entropy solution of~\eqref{eq:criticalconservationlaw} with initial data satisfying $\norm{u_0}_{L^\infty(\R^n)} \leq m$. Suppose also that $f \in C^\infty_\loc(\R;\R^n)$. There exists $\alpha = \alpha(m,n) \in (0,1)$ such that $u \in L^\infty_{t,\loc} C^{2,\alpha}_x(\R^n \times (0,+\infty))$ and
\begin{equation}
  \esssup_{t > 0} \, t \norm{\nabla u(\cdot,t)}_{L^\infty(\R^n)} + t^2 \norm{\nabla^2 u(\cdot,t)}_{L^\infty(\R^n)} + t^{2+\alpha} [\nabla^2 u(\cdot,t)]_{C^\alpha(\R^n)} \les_m 1.
\end{equation}
\end{proposition}
\begin{proof}
The estimate
\begin{equation}
  \label{eq:holderreg}
  t^{\alpha} [u(\cdot,t)]_{C^\alpha(\R^n)} \les_m 1
\end{equation} follows from a direct application of the $L^\infty_x \to C^\alpha_x$ smoothing estimates developed by Silvestre in~\cite[Theorem 1.1]{SilvestreHolderEstimatesAdvectionDiffusion2012} and~\cite{SilvestreHamiltonJacobi2011} for bounded `solutions' of non-local drift-diffusion equations
\begin{equation}
  \label{eq:driftdiffusionequation}
  \p_t u + b \cdot \nabla u + \Lambda u = g,
\end{equation}
where $b, g$ are bounded. Notably, \emph{$b$ may be large and not necessarily divergence free.}
In our situation, $b(x,t) = f'(u(x,t))$ and $g = 0$. The notion of `solution' is in quotation marks because, here, $b$ is allowed to be discontinuous, so the notion of viscosity solution may not be directly applicable.\footnote{This is discussed  in Section~5 of Silvestre's paper~\cite{SilvestreHolderEstimatesAdvectionDiffusion2012}, see also Section~3 of~\cite{SilvestreLipschitz}. Silvestre mentions that, if viscosity solutions are unavailable, then one may justify the estimates at the level of the vanishing viscosity approximation
\begin{equation}
  \label{eq:vanishingviscosity}
  \p_t u^\varepsilon + b^\varepsilon \cdot \nabla u^\varepsilon + \Lambda u^\varepsilon = \varepsilon \Delta u^\varepsilon
\end{equation}
with $\varepsilon \to 0^+$. In principle, this is possible. However, in our setting, the construction in~\cite{AlibaudEntropy2007} was by an operator splitting method, rather than regularization by $\varepsilon \Delta u^\varepsilon$, so we argue differently.} To employ Silvestre's estimates rigorously, one may mollify the initial data, argue at the level of classical solutions, and pass to the limit.

To bootstrap $C^\alpha_x \to C^{1,\alpha}_x$, we apply linear estimates due to Silvestre in~\cite[Theorem 1.1]{SilvestreLipschitz} for the drift-diffusion equation~\eqref{eq:driftdiffusionequation}. It is also possible to proceed more directly, as in Appendix B of~\cite{CaffVassAnnals2010} or in~\cite{ConstVicolGAFA,ConstTarfVicolCMP}. Since $b = f'(u)$ is $\alpha$-H{\"o}lder continuous in $\R^n \times (1/2,1)$ with bounds depending only on $m$, Theorem~1.1 in~\cite{SilvestreLipschitz} gives
\begin{equation}
  \norm{u}_{L^\infty_t C^{1,\alpha}_x(\R^n \times (1/2,1))} \les_m 1.
\end{equation}
Hence, $b = f'(u)$ satisfies the same bounds.
Next, we apply $\p_k$, $1 \leq k \leq n$, to the PDE. This gives
\begin{equation}
  \p_t \p_k u + \Lambda \p_k u + b \cdot \nabla \p_k u = - \p_k b \cdot \nabla u.
\end{equation}
We regard $g = -\p_k b \cdot \nabla u$ as a forcing term belonging to $L^\infty_t C^\alpha_x(\R^n \times (1/2,1))$ Finally, Theorem~1.1 in~\cite{SilvestreLipschitz} gives
\begin{equation}
  \norm{\p_k u}_{L^\infty_t C^{1,\alpha}_x(\R^n \times (3/4,1))} \les_m 1.
\end{equation}
Scaling invariance gives the sharp dependence on $t$. One could also proceed to higher derivatives.
\end{proof}

Consider the linear PDE
\begin{equation}
\label{eq:linearPDE}
  \partial_t u + \Lambda u + b \cdot \nabla u + cu = 0
\end{equation}
where $b \in L^\infty_t C^{1,\alpha}_x(Q_1)$ and $c \in L^\infty_t C^\alpha_x(Q_1)$ with $\norm{b}_{L^\infty_t C^{1,\alpha}_x(Q_1)} + \norm{c}_{L^\infty_t C^\alpha_x(Q_1)} \leq M$.  Here, $Q_T = \R^n \times (0,T)$.


\begin{proposition}[Fundamental solution estimates]
\label{prop:fund:soln}
There exists a continuous function $\Gamma = \Gamma(x,t;y,s)$, $x,y \in \R^n$ and $0 \leq s < t \leq 1$, satisfying the following properties:
\begin{itemize}[leftmargin=*]
	\item (Pointwise upper and lower bounds) For all $0 \leq S \leq s < t \leq T \leq 1$,
\begin{equation}
	\label{eq:fundsolests}
	C_0^{-1} P(x,t;y,s) \leq \Gamma(x,t;y,s) \leq C_0 P(x,t;y,s),
\end{equation}
where $C_0 = C_0(T-S,M) > 0$ and $P$ is the Poisson kernel.
  \item (Maximum principle) If $c \equiv 0$, then $\int_{\R^n} P(x,t;y,s) \, dy = 1$.
\item (Representation formula)
If $w \in L^\infty_{t,x}(Q_1)$, with $w \in L^\infty_{t,\loc} C^{1,\beta}_x(Q_1)$ for some $\beta \in (0,1)$, is a solution of~\eqref{eq:linearPDE} on $Q_1$ and $w(\cdot,t) \wstar w_0$ in $L^\infty(\R^n)$ as $t \to 0^+$, then
\begin{equation}
	w(x,t) = \int_{\R^n} \Gamma(x,t;y,0) w_0(y) \, dy.
\end{equation}
Solutions given by the representation formula belong to the above class.
\end{itemize}
\end{proposition}

Proposition~\ref{prop:fund:soln} was obtained in the paper~\cite{XieZhangHeatKernel2014} of Xie and Zhang by E. E. Levi's parametrix method, \emph{except for uniqueness}, which we sketch below. In~\cite{XieZhangHeatKernel2014}, the authors work with more general assumptions: $b$ in the subcritical space $L^\infty_t C^\alpha_x(Q_1)$ and $c$ in a critical Kato space.

\begin{proof}[Proof of uniqueness]
Let $u_0 \in L^2(\R^n)$. Define 
\begin{equation}
  v(x,t) = \int_{\R^n} \Gamma(x,t;y,0) u_0(y) \, dy.
\end{equation}
Let $L = \Lambda + b \cdot \nabla + c$ and $L^* = \Lambda - b \cdot \nabla + (c - \div b)$.
Under our additional regularity assumptions, it is possible to show that $v$ is a weak solution\footnote{Due to the quite general conditions in~\cite{XieZhangHeatKernel2014}, the authors avoided classical solutions and space-time distributional solutions. Instead, they connect the fundamental solution to the PDE via the `generator' notion.} of the PDE in the sense that
\begin{equation}
  \iint v(x,t) (-\p_t + L^*) \varphi \, dx \, dt = 0
\end{equation}
for all $\varphi \in C^\infty_0(\R^n \times (0,1))$. Additionally, we have $v \in L^\infty_t L^2_x(Q_1)$ and $v \in L^2_{t,\loc} H^{1/2}_x(\R^n \times (0,1])$, among many other spaces, and $v(\cdot,t) \to u_0$ in $L^2(\R^n)$ as $t \to 0^+$. This follows from the pointwise upper bounds of the fundamental solution and its first derivatives (see~\cite[Theorem 1.1 (v)]{XieZhangHeatKernel2014}) and the convergence result in~\cite[Theorem 1.1 (ii)]{XieZhangHeatKernel2014}. One may show, via energy estimates, that the above solution is unique in its class and, additionally, belongs to $C([0,1];L^2(\R^n)) \cap L^2_t H^{1/2}_x(Q_1)$.\footnote{It is important for the energy estimates that $b \in L^\infty_t C^{1/2}_x(Q_1)$.}

Assume now that $u_0 \in L^1 \cap L^\infty(\R^n)$. Then the above solution $v$ also belongs to $L^\infty_t L^1_x \cap L^\infty_{t,x}(Q_1)$. By uniqueness within the energy class, the solution $v$ may be obtained by vanishing viscosity:
\begin{equation}
  \p_t u^\varepsilon + \Lambda u^\varepsilon + b \cdot \nabla u^\varepsilon + c u^\varepsilon = \varepsilon \Delta u^{\varepsilon}.
\end{equation}
According to Silvestre's estimates, we have that $v \in L^\infty_{t,\loc} C^{1,\alpha}_x(\R^n \times (0,1])$ for some $\alpha \in (0,1)$ with estimates depending only on $\norm{u_0}_{L^\infty(\R^n)}$ and the coefficients. By approximation, we have that when $u_0 \in L^\infty(\R^n)$, $v$ satisfies the same \emph{a priori} estimates.

We now demonstrate the following uniqueness theorem by duality: \emph{If $u \in L^\infty_{t,\loc} C^{1,\alpha}_x(\R^n \times (0,1])$ is a solution of the linear PDE~\eqref{eq:linearPDE} with $u(\cdot,t) \wstar 0$ in $L^\infty(\R^n)$ as $t \to 0^+$, then $u \equiv 0$.}\footnote{This argument is modeled off a similar argument in~\cite{albritton2020non} by the first author and Zachary Bradshaw.}

Let $T \in (0,1)$ and $\psi_0 \in L^1 \cap L^\infty(\R^n)$. The above analysis demonstrated that there exists $\psi \in L^\infty_t L^1_x \cap L^\infty_{t,x}(Q_T)$ with $\psi \in L^\infty_{t,\loc} C^{1,\alpha}_x(\R^n \times [0,T))$ and satisfying the adjoint problem
\begin{equation}
  -\p_t \psi + L^* \psi = 0
\end{equation}
with $\psi(T) = \psi_0$. Let $0 < t_0 < t_1 < T$ and $R, \varepsilon > 0$.
Let $\chi \in C^\infty_0(B_2)$ with $\chi \equiv 1$ on $B_1$ and $\chi_R = \chi(x/R)$. Let $\varphi^{t_0,t_1}_\varepsilon$ be a mollification of the indicator function $\mathbf{1}_{(t_0,t_1)}$ at scale $\varepsilon \ll 1$. We test~\eqref{eq:linearPDE} against $\psi \chi_R \varphi^{t_0,t_1}_\varepsilon$ and omit $t_0,t_1,R,\varepsilon$ from the notation as convenient:
\begin{equation}
\begin{aligned}
  &\iint \underbrace{\p_t u + L u}_{= 0 } \psi \chi \varphi \, dx \, dt = \iint \underbrace{-\p_t \psi + L^* \psi}_{= 0} u \varphi \, dx \, dt \\
  &\quad + \iint (-\p_t \varphi) \chi u \psi + \varphi (- b \cdot \nabla \chi) u \psi +  \varphi u [\Lambda, \chi] \psi \, dx \, dt.
  \end{aligned}
\end{equation}
Upon sending $\varepsilon \to 0^+$, we have
\begin{equation}
  \label{eq:intermediateeq}
  \int \chi_R u(x,t_1) \psi(x,t_1) \, dx - \int \chi_R u(x,t_0) \psi(x,t_0) \, dx = \int_{t_0}^{t_1} \int_{\R^n}  b \cdot \nabla \chi_R u \psi +  u [\Lambda, \chi_R] \psi \, dx \, dt
\end{equation}
for a.e. $t_0, t_1 \in (0,1)$. Moreover,~\eqref{eq:intermediateeq} is valid for all $t_0,t_1 \in [0,T]$, since $u \: [0,1] \to L^\infty(\R^n)$ is weak-$\ast$ continuous and $\psi \in C([0,T];L^2(\R^n))$. We focus on $t_0 = 0$ and $t_1 = T$. First, we recall the following estimate for the Calder{\'o}n commutator: $\| [\Lambda,\chi_R] \psi \|_{L^{p'}(\R^n)} \les_p R^{-1} \| \psi \|_{L^{p'}(\R^n)}$ for all $p \in (1,+\infty)$. Additionally, for $|x| \geq 2R$, we have
\begin{equation}
	|[\Lambda,\chi_R] \psi(x,t)| = c_n \left| \int_{\R^n} \frac{\chi_R(y)}{|x-y|^{n+1}} \psi(y,t) \, dy \right| \les_p |x| ^{-(n+1)+n/p} \norm{\psi(\cdot,t)}_{L^{p'}(\R^n)}.
\end{equation}
Hence, $\int_{B_{2R}^c} |[\Lambda,\chi_R] \psi(x,t)| \, dx \les R^{-1+n/p} \norm{\psi(\cdot,t)}_{L^{p'}(\R^n)}$.
By H{\"o}lder's inequality and the above two estimates on $[\Lambda,\chi_R] \psi$, we have
\begin{equation}
  \left| \int_0^T \int_{B_{2R} \cup B_{2R}^c} u [\Lambda, \chi_R] \psi \, dx \, dt \right| \les_p R^{-1+n/p} \norm{u}_{L^\infty_{t,x}(Q_1)} \norm{\psi}_{L^\infty_t L^{p'}_x(Q_1)} \to 0 \text{ as } R \to +\infty
\end{equation}
when $p > n$. The term containing $b \cdot \nabla \chi_R$ is $O(R^{-1})$, since $b, u \in L^\infty_{t,x}(Q_1)$ and $\psi \in L^\infty_t L^1_x(Q_T)$.
 Hence,~\eqref{eq:intermediateeq} becomes
\begin{equation}
  \int u(x,T) \psi_0 \, dx = 0,
\end{equation}
for all $T \in (0,1)$ and $\psi_0 \in L^1 \cap L^\infty(\R^n)$. Therefore, $u \equiv 0$ on $Q_1$.
\end{proof}

\section{Proof of main results}
\label{sec:proofs}

\subsection*{Proof of Theorem~\ref{thm:maintheorem}}

Let $u_0 \in L^\infty$ and $u$ be the corresponding entropy solution. For each $v_0$, we consider the solution $\tilde{u} = u + v$ with initial data $\tilde{u_0}  = u_0 + v_0 \in L^\infty$.

Let $m > 0$ and $\norm{u_0}_{L^\infty}, \norm{\tilde{u_0}}_{L^\infty} \leq m$.

We prove continuity with respect to $v_0$.

\begin{proposition}[Continuity estimate]
  \label{pro:continuityestimate}
Let $1 \leq p \leq q \leq +\infty$. If $v_0 \in L^p(\R^n)$, we have
\begin{equation}
  \label{eq:continuityestimate}
	\norm{v(\cdot,t)}_{L^q(\R^n)} \les_{m,p,q} t^{\frac{n}{q}-\frac{n}{p}} \norm{v_0}_{L^p(\R^n)}.
\end{equation}
\end{proposition}

\begin{proof}[Proof of Proposition~\ref{pro:continuityestimate}]
By scaling invariance, it is enough to demonstrate~\eqref{eq:continuityestimate} with $t=1$.

\textit{Step 1. Propagation of localization}.
First, we demonstrate that, for all $t \in (0,1/2]$, we have
\begin{equation}
  \label{eq:propoflocal}
  \norm{v(\cdot,t)}_{\ell^p_k L^1_x(\R^n)} \les_{m,p} \norm{v_0}_{\ell^p_k L^1_x(\R^n)}.
\end{equation}
Using Proposition~\ref{prop:Alibaud:formula} (Controlled speed of propagation), we have 
\begin{equation}
  \label{eq:initialcalculation}
\begin{aligned}
    \int_{\square(k)} |v(t,x)| dx &\leq \int_{B( k, \sqrt{2n}/2  )} |v(x,t)| dx \\
    &\leq  \int_{B( k, \sqrt{2n}/2 + Lt)} P \ast |v_0|(x) \,dx\\
    &= \sum_{|j| \leq R} \int_{\square(k+j)} P \ast |v_0|(x) \,dx
\end{aligned}
\end{equation}
where $j \in \Z^n$ and $R = R(n,L) > 0$. We apply $\norm{\cdot}_{\ell^p_k(\Z^n)}$ to each side of~\eqref{eq:initialcalculation}. By the triangle inequality, and since the sum in $j$ has only finitely many boxes, we have
\begin{equation}
  \label{eq:reducetoheat}
    \norm{v(\cdot,t)}_{\ell^p_k L^1_x(\R^n)} \les_{R} \norm{P \ast |v_0|}_{\ell^p_k L^1_x(\R^n)}.
\end{equation}
Now Lemma~\ref{lem:smoothing} (Smoothing for the heat equation) with $q_1=q_2=1$ gives~\eqref{eq:propoflocal}.

\textit{Step 2. Smoothing}.
Second, we demonstrate that, for all $t \in (3/4,1]$, we have 
\begin{equation}
  \label{eq:muhsmoothing}
  \norm{v(\cdot,t)}_{L^q(\R^n)} \les_{m,p,q} \norm{v(\cdot,1/2)}_{\ell^p_k L^1_x(\R^n)}.
  \end{equation}
By Proposition~\ref{pro:regularity} (Regularity), $u$ and $\tilde{u}$ belong to $L^\infty_t C^{2,\alpha}_x(\R^n \times (1/2,1))$ with bounds depending only on $m$. Hence, $v = \tilde{u} - u$ belongs to the same space. Let $w(\cdot,t) = v(\cdot,t+1/2)$ when $t \in (0,1/2]$. Let $w_0 = v(\cdot,1/2)$. Then
\begin{equation}
  \p_t w + \div (g(x,t) w) + \Lambda w = 0,
\end{equation}
where
\begin{equation}
  g(x,t) = \frac{f(\tilde{u}) - f(u)}{\tilde{u} - u}
\end{equation}
and
\begin{equation}
  \norm{g}_{L^\infty_t C^{1,\alpha}_x(\R^n \times (0,1/2))} \les_m 1.
\end{equation}
Therefore, we may use the representation formula from Proposition~\ref{prop:fund:soln} (Fundamental solution estimates):
\begin{equation}
w(x,t) = \int_{\mathbb{R}^n } \Gamma(x, t; y,0) w_0(y) \, dy.
\end{equation}
In particular, the pointwise upper bound in Proposition~\ref{prop:fund:soln} gives
\begin{equation}
  |w(x,t)| \les_m \int_{\mathbb{R}^n } P(x-y,t) |w_0|(y) \, dy.
\end{equation}
Then Lemma~\ref{lem:smoothing} (Smoothing for the heat equation) with $q_1 = 1$ and $q_2 = q$, along with the embedding $\ell^p_k L^q_x(\R^n) \into L^q(\R^n)$, gives~\eqref{eq:muhsmoothing}.

Finally, we combine~\eqref{eq:propoflocal} and~\eqref{eq:muhsmoothing} to complete the proof.
\end{proof}

When $v_0 \in L^1(\R^n)$, the propagation of localization comes `for free' from the $L^1$-contraction property, which was shown in~\cite{AlibaudEntropy2007}.

\begin{proof}[Proof of Theorem~\ref{thm:maintheorem}]
Our goal is to demonstrate the $o_{t \to +\infty}(1)$ improvement over the conclusion of Proposition~\ref{pro:continuityestimate} (Continuity estimate) when $p>1$. We approximate $v_0$ strongly in $L^p(\R^n)$ by $v_0^{(k)}$ belonging to $L^1 \cap L^\infty(\R^n)$ and satisfying the decay condition~\eqref{eq:decaycond}, $|v_0^{(k)}| \leq |v_0|$, and $\norm{v_0^{(k)}}_{L^\infty(\R^n)} \leq 2 m$. Let $v^{(k)}$ be the solution  corresponding to the initial data $v_0^{(k)}$. The $o_{t \to +\infty}(1)$ improvement is obvious for $v^{(k)}$, which satisfies a faster decay rate because its initial data belongs to $L^1(\R^n)$. Next, the triangle inequality and Proposition~\ref{pro:continuityestimate} yield
\begin{equation}
\begin{aligned}
  \norm{v(\cdot,t)}_{L^q(\R^n)} &\leq \norm{v^{(k)}(\cdot,t)}_{L^q(\R^n)} + \norm{v(\cdot,t) - v^{(k)}(\cdot,t)}_{L^q(\R^n)} \\
  &\les_m t^{\frac{n}{q} - \frac{n}{p}} o_{t \to +\infty}(1) \underbrace{\norm{v_0^{(k)}}_{L^p(\R^n)}}_{\leq \norm{v_0}_{L^p(\R^n)}} + t^{\frac{n}{q} - \frac{n}{p}} \underbrace{\norm{v_0 - v^{(k)}_0}_{L^p(\R^n)}}_{\to 0 \text{ as } k \to +\infty}.
  \end{aligned}
\end{equation}
This completes the proof.
\end{proof}

\begin{remark}
\label{rmk:alternative}
We record the following alternative proof, due to Hongjie Dong, of Step~1 in Proposition~\ref{pro:continuityestimate}, without the $\ell^p_k L^1_x$ spaces. Consider the adjoint problem to~\eqref{eq:viscouscontinuityeq},
\begin{equation}
	-\partial_t w - b \cdot \nabla w + \Lambda w = 0,
\end{equation}
where
\begin{equation}
	b(x,t) = \int_0^1 f'(\lambda \tilde{u} + (1-\lambda) u) \, d\lambda
\end{equation}
is H{\"o}lder continuous on $\R^n \times [1/2,1]$. Let $h = P(\cdot,1/2) \ast v_0$ and $x_0 \in \R^n$. Testing the PDE~\eqref{eq:viscouscontinuityeq} against the fundamental solution $\Gamma$ of the adjoint problem with pole at $(x_0,1)$, we have
\begin{equation}
	\label{eq:hongjiething}
	|v(x_0,1)| \overset{\eqref{eq:fundsolests}}{\les} \sum_{j=0}^{+\infty} 2^{-j} \barint_{B_{2^j}(x_0)} |v(x,1/2)| \, dx \overset{\eqref{eq:alibaudformula}}{\les} \sum_{j=0}^{+\infty} 2^{-j} \barint_{B_{2^j}(x_0)} |h| \, dx \les (Mh)(x_0),
\end{equation}
where $Mh$ is the maximal function of $h$. One can obtain the $L^p \to L^p$ bound and, more generally, weighted estimates, straightforwardly from~\eqref{eq:hongjiething}.
\end{remark}

\subsection*{$\BV$ convergence and spatial asymptotics}

\begin{proof}[Proof of Theorem~\ref{thm:bvconvergence}]
Let $t_k \to +\infty$ with $t_k \geq 1$. It will be convenient to work with the rescaled solutions
\begin{equation}
  u^{(k)}(y,s) = u(t_k y, t_k s),
\end{equation}
with $\omega^{(k)} = \p_y u^{(k)}$. Then
\begin{equation}
  \norm{\omega^{(k)}(\cdot,1) - \omega^{\SS}(\cdot,1)}_{L^1(\R)} = \norm{\omega(\cdot,t_k) - \omega^{\SS}(\cdot,t_k)}_{L^1(\R)}.
\end{equation}

By Proposition~\ref{pro:regularity} (Regularity) we can bootstrap the decay of $\| u^{(k)}(\cdot, 1) - u^{\SS}( \cdot,1)\|_{L^{\infty}(\R)}$  given by Theorem~ \ref{thm:maintheorem} to get

\begin{equation}
  \norm{\omega^{(k)}(\cdot,1) - \omega^{\SS}(\cdot,1)}_{L^{\infty}(B(R))} \to 0 \text{ as } k \to +\infty
\end{equation}
for all $R \geq 1$. Therefore, it suffices to show that there is no mass of $\omega^{(k)}$ escaping to infinity. Let $R \geq L+10$.
By Alibaud's $\BV$ formula~\eqref{eq:alibaudformulaforbv}, and covering $\R \setminus B(R)$ by an appropriate sequence of balls $B(x_0,1)$, we have
\begin{equation}
  \int_{\R \setminus B(R)} |\omega^{(k)}(x,1)| \, dx \les_m \int_{\R \setminus B(R-L)} P(\cdot,1) \ast |\omega_0^{(k)}| \, dx.
\end{equation}
It is not difficult to show that the quantity on the RHS is $o_{R \to +\infty}(1)$ uniformly in $k$.
\end{proof}

\begin{proof}[Proof of~\eqref{eq:spatialasymptotics}]
First, we remark that $u^{\SS}$ is monotone because the evolution of $\omega$ preserves its sign. This is true at the level of entropy solutions, as can be seen from their construction by the splitting argument in Alibaud's paper.

In the following, we allow the constant $C$ to depend on $u_0 = u_0^{\SS}$ and $\tilde{u}_0$. Let $a,b \in \R$ represent the limits of $u_0$ as $x \to \mp \infty$.

\textit{Step 1. Asymptotics for smooth approximation $\tilde{u}$}.
Let $\tilde{u}_0 \in C^\infty(\R)$ with $\tilde{u}_0 \equiv u_0$ outside of $B_1$. Let $\tilde{u}$ be the corresponding entropy solution, which may be shown to belong to $L^\infty_t C^{2,\alpha}_x(\R^n \times (0,1))$ by combining local-in-time  well-posedness\footnote{See the expository blog post~\cite{taonotes} of Tao on quasilinear well-posedness.} with the estimates in Proposition~\ref{pro:regularity} (Regularity).


By Proposition~\ref{prop:fund:soln} (Fundamental solution estimates), we have
\begin{equation}
  \tilde{u}(x,t) - u_0(x) = \int_{\R^n} P(x,t;y,0) [\tilde{u}_0(y) - u_0(x)] \, dy,
\end{equation}
since $\int P(x,t;y,0) \, dy = 1$ when $c \equiv 0$. Let $x \leq -1$. Hence,
\begin{equation}
  \tilde{u}(x,t) - u_0(x) = \underbrace{\int_{B_1} P(x,t;y,0) [\tilde{u}_0(y) - u_0(x)] \, dy}_{=I_1(x)} + \underbrace{(a-b) \int_{y \geq 1} P(x,t;y,0) \, dy}_{=I_2(x)}.
\end{equation}
Since $[\tilde{u}_0 - u_0(x)]\mathbf{1}_{B_1}$ is compactly supported, we have that $|I_1(x)| \les \la x \ra^{-2}$. On the other hand, when $a \neq b$, we have
\begin{equation}
  C^{-1} \la x \ra^{-1} \leq I_2(x)/(a-b) \leq C \la x \ra^{-1}.
\end{equation}
A similar argument holds for $x \geq 1$. The $I_2$ term will dominate when $|x| \geq C$.

\textit{Step 2. Faster decay for the difference $v$}.
Let $v = u^{\SS} - \tilde{u}$. We will exploit that $v_0 = v(\cdot,1)$ is supported in $B_1$ to demonstrate
\begin{equation}
  |v(x,1)| \leq C \la x \ra^{-2}.
\end{equation}
This will complete the proof, since the $I_2$ term will dominate $|v|$ when $|x| \geq C$.
We follow the scheme of propagation of localization and smoothing as in the proof of Proposition~\ref{pro:continuityestimate}. Let $k \in \Z$ with $|k| \geq 10$. By Alibaud's formula and the decay of the Poisson kernel, we have
\begin{equation}
    \int_{\square(k)} |v(\cdot,1/2)| \, dx \leq C \int_{B(k,\sqrt{2}/2+L)} P(\cdot,1/2) \ast |v_0| \, dx \leq C \la k \ra^{-2}.
\end{equation}
The difference $v$ also satisfies this estimate when $|k| < 10$.
Next, we consider $w(\cdot,t) = v(\cdot,t+1/2)$ and analyze its representation formula when $t \in (1/4,1/2]$:
\begin{equation}
\begin{aligned}
  |w(x,t)| \leq \int_{\R} \Gamma(x,t;y,0) |w_0| \, dy &\leq C \sum_{k \in \Z} \la k \ra^{-2} \norm{\Gamma(x,t;\cdot,0)}_{L^\infty_y(\square(k))} \\
  &\leq C \sum_{k \in \Z} \la k \ra^{-2} \la x - k \ra^{-2} \\
  &\leq C \la x \ra^{-2}.
  \end{aligned}
\end{equation}
The proof is complete.
\end{proof}

\begin{remark}[Rough $f$]
\label{rmk:lipschitzf}
Suppose that $f$ is locally Lipschitz and $n \geq 1$. It is possible to show that, for each $R > 0$, we have
\begin{equation}
  \label{eq:Lipconv}
  \norm{u - u^{\SS}}_{L^\infty(B(Rt))} \to 0 \text{ as } t \to +\infty.
\end{equation}
That is, $u$ converges to $u^{\SS}$ \emph{locally uniformly} in self-similar coordinates $y = x/t$, $s = \log (t/t_0)$ where $t_0 > 0$ is a reference time. Indeed, consider any sequence of rescaled solutions $u^{(k)}$ as above. Since $u_0^{(k)} \to u_0^{\SS}$ in $L^1_\uloc(\R^n)$, Alibaud's formula~\eqref{eq:alibaudformula} gives that $u^{(k)}(\cdot,1)$ converges in $L^1_\uloc(\R^n)$ to $u^{\SS}(\cdot,1)$. By the \emph{a priori} H{\"o}lder estimates~\eqref{eq:holderreg} and the Ascoli-Arzel{\'a} theorem,\footnote{To justify~\eqref{eq:holderreg} with Lipschitz $f$, one could mollify $f$ or apply a parabolic regularization $\varepsilon \Delta u^\varepsilon$, justify the estimates at the regularized level, and pass to the limit.} $u^{(k)}(\cdot,1)$ converges in $L^\infty_\loc(\R^n)$, and its limit must be $u^{\SS}(\cdot,1)$.

If $n = 1$ and $u_0 \in \BV(\R)$, then we may choose $R = +\infty$ in~\eqref{eq:Lipconv}, since the $\BV(\R)$ norm manages the behavior in $L^\infty(\R \setminus B_R)$ for $R \gg 1$ according to Alibaud's $\BV$ formula~\eqref{eq:alibaudformulaforbv}. If $f \in C^{1,\alpha}_\loc(\R)$, then it is possible to upgrade to $\BV(\R)$ convergence, since Silvestre's estimates in~\cite{SilvestreLipschitz} allow the solution to be bootstrapped from $C^\alpha_x(\R)$ to $C^{1,\alpha}_x(\R)$.
\end{remark}

\subsubsection{Acknowledgments}  DA thanks Vladim{\'i}r {\v S}ver{\'a}k for mentioning this problem and helpful discussions. We thank the referee for his/her work on the paper. We also thank Montie Avery, Vlad Vicol, and Jan Burczak for comments on a preliminary version, and Hongjie Dong for the alternative proof in Remark~\ref{rmk:alternative}. DA was supported by the NDSEG Fellowship and the NSF Postdoctoral Fellowship Grant No. 2002023. RB was supported by the NSF Graduate Fellowship Grant No. 1839302. 

\bibliographystyle{plain} 
\bibliography{bibliography}

\begin{thebibliography}{10}

\bibitem{AsymptoticTraveling}
Franz Achleitner and Yoshihiro Ueda.
\newblock Asymptotic stability of traveling wave solutions for nonlocal viscous
  conservation laws with explicit decay rates.
\newblock {\em J. Evol. Equ.}, 18(2):923--946, 2018.

\bibitem{albritton2020non}
Dallas Albritton and Zachary Bradshaw.
\newblock Non-decaying solutions to the critical surface quasi-geostrophic
  equations with symmetries.
\newblock {\em arXiv preprint arXiv:2011.10856}, 2020.

\bibitem{AlibaudEntropy2007}
Natha\"{e}l Alibaud.
\newblock Entropy formulation for fractal conservation laws.
\newblock {\em J. Evol. Equ.}, 7(1):145--175, 2007.

\bibitem{AlibaudOccurrenceNonAppearance}
Natha\"{e}l Alibaud, J\'{e}r\^{o}me Droniou, and Julien Vovelle.
\newblock Occurrence and non-appearance of shocks in fractal {B}urgers
  equations.
\newblock {\em J. Hyperbolic Differ. Equ.}, 4(3):479--499, 2007.

\bibitem{AsymptoticPropertiesEntropySolutions}
Nathael Alibaud, Cyril Imbert, and Grzegorz Karch.
\newblock Asymptotic properties of entropy solutions to fractal {B}urgers
  equation.
\newblock {\em SIAM J. Math. Anal.}, 42(1):354--376, 2010.

\bibitem{BeckWayne}
Margaret Beck and C.~Eugene Wayne.
\newblock Using global invariant manifolds to understand metastability in the
  {B}urgers equation with small viscosity.
\newblock {\em SIAM J. Appl. Dyn. Syst.}, 8(3):1043--1065, 2009.

\bibitem{BilerFractalBurgers1998}
Piotr Biler, Tadahisa Funaki, and Wojbor~A. Woyczynski.
\newblock Fractal {B}urgers equations.
\newblock {\em J. Differential Equations}, 148(1):9--46, 1998.

\bibitem{bradshaw2020local}
Zachary Bradshaw and Tai-Peng Tsai.
\newblock Local energy solutions to the {N}avier-{S}tokes equations in {W}iener
  amalgam spaces.
\newblock {\em arXiv preprint arXiv:2008.09204}, 2020.

\bibitem{BricmontKupiainenLin}
J.~Bricmont, A.~Kupiainen, and G.~Lin.
\newblock Renormalization group and asymptotics of solutions of nonlinear
  parabolic equations.
\newblock {\em Comm. Pure Appl. Math.}, 47(6):893--922, 1994.

\bibitem{BurczakKellerSegel}
Jan Burczak and Rafael Granero-Belinch\'{o}n.
\newblock Critical {K}eller-{S}egel meets {B}urgers on {$\Bbb S^1$}: large-time
  smooth solutions.
\newblock {\em Nonlinearity}, 29(12):3810--3836, 2016.

\bibitem{CaffVassAnnals2010}
Luis~A. Caffarelli and Alexis Vasseur.
\newblock Drift diffusion equations with fractional diffusion and the
  quasi-geostrophic equation.
\newblock {\em Ann. of Math. (2)}, 171(3):1903--1930, 2010.

\bibitem{ChanCzubakBurgers2010}
Chi~Hin Chan and Magdalena Czubak.
\newblock Regularity of solutions for the critical {$N$}-dimensional {B}urgers'
  equation.
\newblock {\em Ann. Inst. H. Poincar\'{e} Anal. Non Lin\'{e}aire},
  27(2):471--501, 2010.

\bibitem{ChanCzubakSilvestreEventualRegularity2010}
Chi~Hin Chan, Magdalena Czubak, and Luis Silvestre.
\newblock Eventual regularization of the slightly supercritical fractional
  {B}urgers equation.
\newblock {\em Discrete Contin. Dyn. Syst.}, 27(2):847--861, 2010.

\bibitem{ChernLiu}
I-Liang Chern and Tai-Ping Liu.
\newblock Convergence to diffusion waves of solutions for viscous conservation
  laws.
\newblock {\em Comm. Math. Phys.}, 110(3):503--517, 1987.

\bibitem{ChmajTravelling}
Adam Chmaj.
\newblock Existence of travelling waves in the fractional {B}urgers equation.
\newblock {\em Bull. Aust. Math. Soc.}, 97(1):102--109, 2018.

\bibitem{ConstTarfVicolCMP}
Peter Constantin, Andrei Tarfulea, and Vlad Vicol.
\newblock Long time dynamics of forced critical {SQG}.
\newblock {\em Comm. Math. Phys.}, 335(1):93--141, 2015.

\bibitem{ConstVicolGAFA}
Peter Constantin and Vlad Vicol.
\newblock Nonlinear maximum principles for dissipative linear nonlocal
  operators and applications.
\newblock {\em Geom. Funct. Anal.}, 22(5):1289--1321, 2012.

\bibitem{DabkowskiEventualReg}
Michael Dabkowski.
\newblock Eventual regularity of the solutions to the supercritical dissipative
  quasi-geostrophic equation.
\newblock {\em Geom. Funct. Anal.}, 21(1):1--13, 2011.

\bibitem{Changhui}
Tam Do, Alexander Kiselev, Lenya Ryzhik, and Changhui Tan.
\newblock Global regularity for the fractional {E}uler alignment system.
\newblock {\em Arch. Ration. Mech. Anal.}, 228(1):1--37, 2018.

\bibitem{DongDuLiFiniteTimeFractal}
Hongjie Dong, Dapeng Du, and Dong Li.
\newblock Finite time singularities and global well-posedness for fractal
  {B}urgers equations.
\newblock {\em Indiana Univ. Math. J.}, 58(2):807--821, 2009.

\bibitem{GallayWayne}
Thierry Gallay and C.~Eugene Wayne.
\newblock Global stability of vortex solutions of the two-dimensional
  {N}avier-{S}tokes equation.
\newblock {\em Comm. Math. Phys.}, 255(1):97--129, 2005.

\bibitem{guillod2017numerical}
Julien Guillod and Vladim{\'\i}r {\v{S}}ver{\'a}k.
\newblock Numerical investigations of non-uniqueness for the {N}avier-{S}tokes
  initial value problem in borderline spaces.
\newblock {\em arXiv preprint arXiv:1704.00560}, 2017.

\bibitem{Oleinik}
A.~M. Il{\'i}n and O.~A. Ole{\u i}nik.
\newblock Behavior of solutions of the {C}auchy problem for certain quasilinear
  equations for unbounded increase of the time.
\newblock {\em Dokl. Akad. Nauk SSSR}, 120:25--28, 1958.

\bibitem{Iwabuchi1}
Tsukasa Iwabuchi.
\newblock Global solutions for the critical {B}urgers equation in the {B}esov
  spaces and the large time behavior.
\newblock {\em Ann. Inst. H. Poincar\'{e} Anal. Non Lin\'{e}aire},
  32(3):687--713, 2015.

\bibitem{Iwabuchi2}
Tsukasa Iwabuchi.
\newblock Analyticity and large time behavior for the {B}urgers equation and
  the quasi-geostrophic equation, the both with the critical dissipation.
\newblock {\em Ann. Inst. H. Poincar\'{e} Anal. Non Lin\'{e}aire},
  37(4):855--876, 2020.

\bibitem{JiaSverakIllposed}
Hao Jia and Vladim{\'\i}r {\v S}ver\'{a}k.
\newblock Are the incompressible 3d {N}avier-{S}tokes equations locally
  ill-posed in the natural energy space?
\newblock {\em J. Funct. Anal.}, 268(12):3734--3766, 2015.

\bibitem{Kapitula}
Christopher K. R.~T. Jones, Robert Gardner, and Todd Kapitula.
\newblock Stability of travelling waves for nonconvex scalar viscous
  conservation laws.
\newblock {\em Comm. Pure Appl. Math.}, 46(4):505--526, 1993.

\bibitem{KarchConvergenceRarefactionWaves}
Grzegorz Karch, Changxing Miao, and Xiaojing Xu.
\newblock On convergence of solutions of fractal {B}urgers equation toward
  rarefaction waves.
\newblock {\em SIAM J. Math. Anal.}, 39(5):1536--1549, 2008.

\bibitem{DiffusiveNWaves}
Yong~Jung Kim and Athanasios~E. Tzavaras.
\newblock Diffusive {$N$}-waves and metastability in the {B}urgers equation.
\newblock {\em SIAM J. Math. Anal.}, 33(3):607--633, 2001.

\bibitem{VariationsThemeCaffVass2009}
A.~Kiselev and F.~Nazarov.
\newblock A variation on a theme of {C}affarelli and {V}asseur.
\newblock {\em Zap. Nauchn. Sem. S.-Peterburg. Otdel. Mat. Inst. Steklov.
  (POMI)}, 370(Kraevye Zadachi Matematichesko\u{\i} Fiziki i Smezhnye Voprosy
  Teorii Funktsi\u{\i}. 40):58--72, 220, 2009.

\bibitem{KiselevNazarovVolbergInventiones2007}
A.~Kiselev, F.~Nazarov, and A.~Volberg.
\newblock Global well-posedness for the critical 2{D} dissipative
  quasi-geostrophic equation.
\newblock {\em Invent. Math.}, 167(3):445--453, 2007.

\bibitem{KiselevNonlocalMaximumPrinciples}
Alexander Kiselev.
\newblock Nonlocal maximum principles for active scalars.
\newblock {\em Adv. Math.}, 227(5):1806--1826, 2011.

\bibitem{KiselevFractalBurgers2008}
Alexander Kiselev, Fedor Nazarov, and Roman Shterenberg.
\newblock Blow up and regularity for fractal {B}urgers equation.
\newblock {\em Dyn. Partial Differ. Equ.}, 5(3):211--240, 2008.

\bibitem{MaekawaMiura-drift}
Yasunori Maekawa and Hideyuki Miura.
\newblock On fundamental solutions for non-local parabolic equations with
  divergence free drift.
\newblock {\em Adv. Math.}, 247:123--191, 2013.

\bibitem{Miller}
Joel~C. Miller and Andrew~J. Bernoff.
\newblock Rates on convergence to self-similar solutions of {B}urgers'
  equation.
\newblock {\em Stud. Appl. Math.}, 111(1):29--40, 2003.

\bibitem{Sattinger}
D.~H. Sattinger.
\newblock On the stability of waves of nonlinear parabolic systems.
\newblock {\em Advances in Math.}, 22(3):312--355, 1976.

\bibitem{SilvestreEventual}
Luis Silvestre.
\newblock Eventual regularization for the slightly supercritical
  quasi-geostrophic equation.
\newblock {\em Ann. Inst. H. Poincar\'{e} Anal. Non Lin\'{e}aire},
  27(2):693--704, 2010.

\bibitem{SilvestreHamiltonJacobi2011}
Luis Silvestre.
\newblock On the differentiability of the solution to the {H}amilton-{J}acobi
  equation with critical fractional diffusion.
\newblock {\em Adv. Math.}, 226(2):2020--2039, 2011.

\bibitem{SilvestreHolderEstimatesAdvectionDiffusion2012}
Luis Silvestre.
\newblock H\"{o}lder estimates for advection fractional-diffusion equations.
\newblock {\em Ann. Sc. Norm. Super. Pisa Cl. Sci. (5)}, 11(4):843--855, 2012.

\bibitem{SilvestreLipschitz}
Luis Silvestre.
\newblock On the differentiability of the solution to an equation with drift
  and fractional diffusion.
\newblock {\em Indiana Univ. Math. J.}, 61(2):557--584, 2012.

\bibitem{Stokols}
Logan~F. Stokols and Alexis~F. Vasseur.
\newblock H\"{o}lder regularity up to the boundary for critical {SQG} on
  bounded domains.
\newblock {\em Arch. Ration. Mech. Anal.}, 236(3):1543--1591, 2020.

\bibitem{taonotes}
Terence Tao.
\newblock Quasilinear well-posedness, 2010.
\newblock URL:
  https://terrytao.wordpress.com/2010/02/21/quasilinear-well-posedness/.
  Last visited on September 19, 2020.

\bibitem{XieZhangHeatKernel2014}
Longjie Xie and Xicheng Zhang.
\newblock Heat kernel estimates for critical fractional diffusion operators.
\newblock {\em Studia Math.}, 224(3):221--263, 2014.

\bibitem{ZumbrunHoward}
Kevin Zumbrun and Peter Howard.
\newblock Pointwise semigroup methods and stability of viscous shock waves.
\newblock {\em Indiana Univ. Math. J.}, 47(3):741--871, 1998.

\end{thebibliography}

\end{document}